\documentclass[draft,10pt]{article}%
\usepackage{amscd}
\usepackage{amsmath}
\usepackage{amsfonts}
\usepackage{amssymb}
\usepackage{graphicx}
\usepackage[T1]{fontenc}
\usepackage[margin=0.5in]{geometry}
\usepackage{enumerate}
\usepackage[colorlinks=true]{hyperref}%
\providecommand{\U}[1]{\protect\rule{.1in}{.1in}}
\hypersetup{urlcolor=blue, citecolor=cyan}
\hypersetup{citecolor=blue}
\providecommand{\U}[1]{\protect\rule{.1in}{.1in}}
\providecommand{\U}[1]{\protect\rule{.1in}{.1in}}

\numberwithin{equation}{section}
\newtheorem{theorem}{Theorem}[section]

\newtheorem{lemma}[theorem]{Lemma}
\newtheorem{proposition}[theorem]{Proposition}

\newtheorem{rem}[theorem]{Remark}

\newenvironment{proof}[1][Proof]{\noindent\textbf{#1.} }{\ \rule{0.5em}{0.5em}}

\newcommand\LINK[1]{\href{#1}{(link: \nolinkurl{#1})}}
\def\p2{\mathcal A_{\Phi,2\pi}(B)}
\def\0p2{\mathcal A_{\Phi,2\pi}(0)}
\def\sp2{\mathcal A_{\Phi,2\pi,\hbox{\rm SR}}(B)}
\def\beq{\begin{equation}}
\def\ene{\end{equation}}

\def\qed{\ifhmode\unskip\nobreak\fi\ifmmode\ifinner
\else\hskip5pt\fi\fi\hbox{\hskip5pt\vrule width4pt height6pt
depth1.5pt\hskip1pt}}

\def\+out{x^{\rm out}}
\begin{document}

\title{Estimates in the modulation spaces for the Dirac equation with potential}
\author{Keiichi Kato\thanks{Electronic Mail: kato@ma.kagu.tus.ac.jp}\\Department of Mathematics, Faculty of Science, \\Tokyo University of Science, Kagurazaka 1-3,\\Shinjuku-ku, Tokyo 162-8601, Japan. \\and \\Ivan Naumkin\thanks{Electronic Mail: ivan.naumkin@unice.fr} \thanks{Ivan
Naumkin thanks the project ERC-2014-CdG 646.650 SingWave for its financial
support, and the Laboratoire J.A. Dieudonn\'e of the Universit\'e de Nice
Sophia-Antipolis for its kind hospitality.}\\Laboratoire J.A. Dieudonn\'{e}, UMR CNRS 7351,\\Universit\'{e} de Nice Sophia-Antipolis,\\Parc Valrose, 06108 Nice Cedex 02, France.}
\date{}
\maketitle

\begin{abstract}
In the present paper we obtain estimates in the modulation spaces for the
solutions to the Dirac equation with quadratic and sub-quadratic potentials.
We derive a representation for the Dirac operator that permits to solve
approximately the perturbed Dirac equation and to obtain the desired estimates
for the solution.

\end{abstract}

\section{\bigskip Introduction.}

In this paper, we aim to obtain estimates in the modulation spaces for the
solutions to the Cauchy problem for the Dirac equation%

\begin{equation}
\left\{
\begin{array}
[c]{c}%
i\partial_{t}\psi\left(  t,x\right)  =H_{0}\psi\left(  t,x\right)
+\mathbf{V}\left(  t,x\right)  \psi\left(  t,x\right)  ,\text{ \ }%
(t,x)\in\mathbb{R\times R}^{3},\\
\psi\left(  0,x\right)  =\psi_{0}\left(  x\right)  ,\text{ }x\in\mathbb{R}%
^{3},
\end{array}
\right.  \label{Dirac}%
\end{equation}
where $i=\sqrt{-1},$ $\psi\left(  t,x\right)  =\left(  \psi_{1}\left(
t,x\right)  ,\psi_{2}\left(  t,x\right)  ,\psi_{3}\left(  t,x\right)
,\psi_{4}\left(  t,x\right)  \right)  ^{T}\in\mathbb{C}^{4}$ is a four-spinor
field,%
\begin{equation}
H_{0}=-i\alpha\cdot\nabla+m\beta, \label{D1}%
\end{equation}
is the free Dirac operator with $m$ - the mass of the particle, $\alpha
=(\alpha_{1},\alpha_{2},\alpha_{3}),$ and $\alpha_{j},$ $j=1,2,3,4,$ are
$4\times4$ Hermitian matrices that satisfy the relation:%
\[
\alpha_{j}\alpha_{k}+\alpha_{k}\alpha_{j}=2\delta_{jk},\text{ }1\leq
j,k\leq4,
\]
where $\delta_{jk}$ denotes the Kronecker symbol. The standard choice of
$\alpha_{j}$ is (\cite{Thaller}):
\[
\alpha_{j}=%
\begin{pmatrix}
0 & \sigma_{j}\\
\sigma_{j} & 0
\end{pmatrix}
,\text{ \ }1\leq j\leq3,\text{ \ \ \ \ }\alpha_{4}=%
\begin{pmatrix}
I_{2} & 0\\
0 & -I_{2}%
\end{pmatrix}
=\beta,
\]
($I_{n}$ is the $n\times n$ unit matrix) and
\[
\sigma_{1}=%
\begin{pmatrix}
0 & 1\\
1 & 0
\end{pmatrix}
,\sigma_{2}=%
\begin{pmatrix}
0 & -i\\
i & 0
\end{pmatrix}
,\sigma_{3}=%
\begin{pmatrix}
1 & 0\\
0 & -1
\end{pmatrix}
\]
are the Pauli matrices. The potential $\mathbf{V}\left(  t,x\right)  \in
C^{\infty}\left(  \mathbb{R\times R}^{3}\right)  $ is a $\left(
4\times4\right)  $-matrix valued function which entries $\mathbf{V}%
_{jk}\left(  t,x\right)  ,$\ $1\leq j,k\leq4,$ for all multi-indices $\alpha$
with $\left\vert \alpha\right\vert \geq2$ or $\left\vert \alpha\right\vert
\geq1,$ satisfy the estimates%

\begin{equation}
\left\vert \partial_{x}^{\alpha}\mathbf{V}_{jk}\left(  t,x\right)  \right\vert
\leq C_{\alpha},\text{ \ }1\leq j,k\leq4. \label{t23}%
\end{equation}

The usual framework for equation (\ref{Dirac}) is a $L^{2}$ based space, as
for instance, the Sobolev spaces $H^{s}.$ The question arises if it is
possible to remove the $L^{2}$ constraint and consider equation (\ref{Dirac})
in functional spaces which are not $L^{2}$ based. In the case when
$H_{0}=-\Delta$ several approaches were used in order to answer this question.
For example, the local well-posedness of the NLS\ equation was studied on
Zhidkov spaces in \cite{Galo}. Different spaces of infinite mass were
introduced in \cite{Zhou} and \cite{Simao} to study the well-posedness problem
for the NLS\ equation. We also mention the papers \cite{Gerard} and
\cite{Vega} that consider other frameworks which are not $L^{2}$ based.

In general, for a Fourier multiplier $e^{itH},$ the main drawback in working
in Lebesgue spaces $L^{p}$ is that $e^{itH}$ may be unbounded. This means that
the initial properties are not preserved by the time evolution. In the case of
unimodular Fourier multipliers $e^{iH}$ with general symbols $e^{i\left\vert
\xi\right\vert ^{\alpha}},$ where $\alpha\in\lbrack0,2]$, the modulation
spaces $M_{r,s}^{p,q}$ (see Section \ref{S2.3} for the definition of these
spaces) have resulted to be an alternative for the study of $e^{iH}.$ It was
shown in \cite{Benyi2} that such multipliers are bounded on all modulation
spaces, even if they are unbounded on usual $L^{p}$-spaces. There exist a
large literature concerning the modulation spaces and their applications to
the Schr\"{o}dinger equation or other equations, such as the wave equation or
the Klein-Gordon equation. For example, we mention the works \cite{Benyi}%
-\cite{Cordero5}, \cite{Kato1}-\cite{Miyachi}, \cite{Tomita}, \cite{Wang1}%
-\cite{Wei}, and the references cited therein.\ As far as we know, there are
no papers concerning the Dirac equation in the framework of the modulation
spaces. We pretend to fill this gap by proving some estimates in modulation
spaces for the solutions of equation (\ref{Dirac}).

Let us first recall some known results. In the case of the free
Schr\"{o}dinger operator $H_{0}=-\dfrac{1}{2}\Delta$ and $\mathbf{V}\equiv0$,
estimates on the modulation spaces for the solutions to the corresponding
Cauchy problem were obtained in \cite{Benyi2}, \cite{Wang} and \cite{Wang2}%
.\ More precisely, the following was proved. Consider the Schr\"{o}dinger
equation%
\begin{equation}
\left\{
\begin{array}
[c]{c}%
i\partial_{t}u\left(  t,x\right)  =-\dfrac{1}{2}\Delta u\left(  t,x\right)
+V\left(  t,x\right)  u\left(  t,x\right)  ,\text{ \ }(t,x)\in\mathbb{R\times
R}^{N},\\
u\left(  0,x\right)  =u_{0}\left(  x\right)  ,\text{ }x\in\mathbb{R}^{N},
\end{array}
\right.  \label{Shrodinger}%
\end{equation}
where $u\left(  t,x\right)  $ is a complex-valued function of $(t,x)\in
\mathbb{R\times R}^{N},$ $V\left(  t,x\right)  $ is a real-valued function of
$(t,x)\in\mathbb{R\times R}^{N}$ and $u_{0}\left(  x\right)  $ is a
complex-valued function of $x\in\mathbb{R}^{N}.$

\textbf{Theorem A. \ \ }\textit{i) (See \cite{Benyi2}.) Let }$1\leq
p,q\leq\infty.$\textit{ Suppose that }$V\left(  t,x\right)  \equiv0.$\textit{
Then, there exists a positive constant }$C$\textit{ such that }%
\[
\left\Vert u\left(  t,\cdot\right)  \right\Vert _{M^{p,q}}\leq C\left(
1+\left\vert t\right\vert \right)  ^{N/2}\left\Vert u_{0}\right\Vert
_{M^{p,q}},\text{ }u_{0}\in\mathcal{S}\left(  \mathbb{R}^{N}\right)  ,
\]
\textit{for all }$t\in R$\textit{, where }$u\left(  t,x\right)  $\textit{ is
the solution of (\ref{Shrodinger})}$.$

\ \ \ \ \ \ \ \ \ \ \ \ \ \ \ \ \ \ \ \textit{ii) (See \cite{Wang}.) Let
}$2\leq p<\infty,$\textit{ }$1\leq q<\infty,$\textit{ }$1/p+1/p^{\prime}%
=1.$\textit{ Suppose that }$V\left(  t,x\right)  \equiv0.$\textit{ Then, there
exists positive constants }$C$\textit{ and }$C^{\prime}$\textit{ such that}%
\[
\left\Vert u\left(  t,\cdot\right)  \right\Vert _{M^{p,q}}\leq C\left(
1+\left\vert t\right\vert \right)  ^{-N\left(  1/2-1/p\right)  }\left\Vert
u_{0}\right\Vert _{M^{p^{\prime},q}},\text{ }u_{0}\in\mathcal{S}\left(
\mathbb{R}^{N}\right)  ,
\]
\textit{and}%
\[
\left\Vert u\left(  t,\cdot\right)  \right\Vert _{M^{p,q}}\leq C^{\prime
}\left(  1+\left\vert t\right\vert \right)  ^{N\left(  1/2-1/p\right)
}\left\Vert u_{0}\right\Vert _{M^{p,q}},\text{ }u_{0}\in\mathcal{S}\left(
\mathbb{R}^{N}\right)  ,
\]
\textit{for all }$t\in R$\textit{, where }$u\left(  t,x\right)  $\textit{ is
the solution of (\ref{Shrodinger})}$.$

A new representation for the Schr\"{o}dinger operator via the wave packet
transform was derived in \cite{Kato1}-\cite{Kato2} and used to study equation
\textit{(\ref{Shrodinger})}\ in the context of the modulation spaces. In
particular, the results of Theorem A\textbf{ }were proved in \cite{Kato2} by
using this representation. As it was showed in \cite{Kato3} and \cite{Kato4}%
,\textit{ }the approach of \cite{Kato1}-\cite{Kato2} may be applied to study
equation (\ref{Shrodinger}) with quadratic and sub-quadratic potentials to
estimate its solution in the modulation spaces. The following results are due
to \cite{Kato3} and \cite{Kato4}. (Below we emphasize the dependence of the
modulation spaces $M_{\phi}^{p,q}$ on the window $\phi,$ see Section
\ref{S2.3} for more details.)

\textbf{Theorem B. \textit{ }}\textit{(See \cite{Kato3}.) Let }$1\leq
p\leq\infty$\textit{ and }$\phi_{0}\in S\left(  \mathbb{R}^{N}\right)
\setminus\left\{  0\right\}  .$\textit{ Suppose that }$V\left(  t,x\right)
=\pm\frac{1}{2}\left\vert x\right\vert ^{2}.$\textit{ Then, }%
\[
\left\Vert u\left(  t,\cdot\right)  \right\Vert _{M_{\phi\left(
t,\cdot\right)  }^{p,p}}=\left\Vert u_{0}\right\Vert _{M_{\phi_{0}}^{p,p}%
},\text{ }u_{0}\in\mathcal{S}\left(  \mathbb{R}^{N}\right)  ,
\]
\textit{for all }$t\in R,$\textit{ where }$u\left(  t,x\right)  $\textit{ and
}$\phi\left(  t,x\right)  $\textit{ are the solutions of (\ref{Shrodinger})
with }$u\left(  0,x\right)  =u_{0}\left(  x\right)  $\textit{ and }%
$\phi\left(  t,x\right)  =\phi_{0}\left(  x\right)  ,$\textit{ respectively.}

\textbf{Theorem C. \textit{ }}\textit{(See \cite{Kato4}.) i) Let }$1\leq
p\leq\infty$\textit{, }$\phi_{0}\in\mathcal{S}\left(  \mathbb{R}^{N}\right)
\setminus\left\{  0\right\}  $ and $T>0.\ $Set $\phi\left(  t,x\right)
:=e^{it\frac{\Delta}{2}}\phi_{0}.$\textit{ Suppose that }$V\left(  t,x\right)
$ is a real-valued function satisfying (\ref{t23}) for $\left\vert
\alpha\right\vert \geq2.\ $Then,\textit{ there exists a positive constant
}$C_{T}$\textit{ such that}%
\begin{equation}
\left\Vert u\left(  t,\cdot\right)  \right\Vert _{M_{\phi\left(
t,\cdot\right)  }^{p,p}}\leq C_{T}\left\Vert u_{0}\right\Vert _{M_{\phi_{0}%
}^{p,p}},\text{ }u_{0}\in\mathcal{S}\left(  \mathbb{R}^{N}\right)  ,
\label{t4}%
\end{equation}
uniformly for $t\in\lbrack-T,T],$ \textit{where }$u\left(  t,x\right)
$\textit{ is the solution of (\ref{Shrodinger}) in }$C\left(  \mathbb{R}%
;L^{2}\left(  \mathbb{R}^{N}\right)  \right)  $\textit{.}

\ \ \ \ \ \ \ \ \ \ \ \ \ \ \ \ \ \ \ \ \ \ \ ii) \textit{Let }$1\leq
p,q\leq\infty$\textit{, }$\phi_{0}\in S\left(  \mathbb{R}^{N}\right)
\setminus\left\{  0\right\}  $\textit{ and }$T>0.$\textit{ Set }$\phi\left(
t,x\right)  :=e^{it\frac{\Delta}{2}}\phi_{0}.$\textit{ Suppose that }$V\left(
t,x\right)  $\textit{ is a real-valued function satisfying (\ref{t23}) for
}$\left\vert \alpha\right\vert \geq1.$\textit{ Then, there exists a positive
constant }$C_{T}$\textit{ such that}%
\begin{equation}
\left\Vert u\left(  t,\cdot\right)  \right\Vert _{M_{\phi\left(
t,\cdot\right)  }^{p,q}}\leq C_{T}\left\Vert u_{0}\right\Vert _{M_{\phi_{0}%
}^{p,q}},\text{ }u_{0}\in\mathcal{S}\left(  \mathbb{R}^{N}\right)  ,
\label{t35}%
\end{equation}
\textit{uniformly for }$t\in\lbrack-T,T],$\textit{ where }$u\left(
t,x\right)  $\textit{ is the solution of (\ref{Shrodinger}) in }$C\left(
\mathbb{R};L^{2}\left(  \mathbb{R}^{N}\right)  \right)  $\textit{.}

We observe that the approach of \cite{Kato1}-\cite{Kato3} was also used in
\cite{Wei} to prove a result similar to Theorem C for more general
Schr\"{o}dinger-type equations with time-dependent quadratic or sub-quadratic potentials.

\subsection*{Main results.}

We now present our main results. First, we consider the free case, similar to
Theorem A.

\begin{theorem}
\label{TDiracfree}i) Let $1\leq p,q\leq\infty,$ and $T>0.$ Suppose that
$\mathbf{V}\left(  t,x\right)  =0,$ for all $(t,x)\in\mathbb{R\times R}^{3}.$
Then, the solution $\psi\left(  t,x\right)  $ of (\ref{Dirac}) satisfies%
\begin{equation}
\left\Vert \psi\left(  t,\cdot\right)  \right\Vert _{M^{p,q}}\leq
C_{T}\left\Vert \psi_{0}\right\Vert _{M^{p,q}},\text{ }\psi_{0}\in
\mathcal{S}\left(  \mathbb{R}^{N}\right)  , \label{t54}%
\end{equation}
uniformly for $t\in\lbrack-T,T]$.

ii) Let $0<q<\infty.$ Suppose that $\mathbf{V}\left(  t,x\right)  =0,$ for all
$(t,x)\in\mathbb{R\times R}^{3}.$ Then
\begin{equation}
\left\Vert \psi\left(  t,\cdot\right)  \right\Vert _{M_{0,-2\sigma}^{p,q}}\leq
C\left\langle t\right\rangle ^{-\theta\lbrack3/2-3/p]}\left\Vert \psi
_{0}\right\Vert _{M^{p^{\prime},q}},\text{ }\psi_{0}\in\mathcal{S}\left(
\mathbb{R}^{N}\right)  ,\text{\ } \label{t53}%
\end{equation}
where $2\leq p<\infty,$ $\frac{1}{p}+$\ $\frac{1}{p^{\prime}}=1,$ $\theta
\in\lbrack0,1]$ and $2\sigma=5\theta\left(  \frac{1}{2}-\frac{1}{p}\right)  .$
\end{theorem}

Now, we present the estimates analogous to Theorem C for the equation
(\ref{Dirac}) with quadratic and sub-quadratic potentials.

\begin{theorem}
\label{TDirac}i) Let $1\leq p\leq\infty,$ and $T>0.$ Suppose that the
potential in problem (\ref{Dirac}) decomposes as $\mathbf{V}\left(
t,x\right)  =Q\left(  t,x\right)  I_{4}+\mathbf{V}_{2}\left(  t,x\right)  ,$
where $Q\left(  t,x\right)  \in C^{\infty}\left(  \mathbb{R}^{3}%
\times\mathbb{R}^{3}\right)  $ is a real-valued function satisfying
(\ref{t23}) for $\left\vert \alpha\right\vert \geq2$ and $\mathbf{V}%
_{2}\left(  t,x\right)  $ is a $\left(  4\times4\right)  -$matrix-valued
function whose components verify (\ref{t23}) for $\left\vert \alpha\right\vert
\geq0.$ Then,
\begin{equation}
\left\Vert \psi\left(  t,\cdot\right)  \right\Vert _{M^{p,p}}\leq
C_{T}\left\Vert \psi_{0}\right\Vert _{M^{p,p}},\text{ }\psi_{0}\in
\mathcal{S}\left(  \mathbb{R}^{N}\right)  , \label{t55}%
\end{equation}
uniformly for $t\in\lbrack-T,T].$

ii) Let $1\leq p,q\leq\infty,$ and $T>0.$ Suppose that $\mathbf{V}\left(
t,x\right)  =\mathbf{V}_{1}\left(  t,x\right)  +\mathbf{V}_{2}\left(
t,x\right)  ,$ where $\mathbf{V}_{1}\left(  t,x\right)  =\{\left(
\mathbf{V}_{1}\right)  _{jk}\left(  t,x\right)  \},$ $1\leq j,k\leq4,$ is a
matrix-valed function with $\left(  \mathbf{V}_{1}\right)  _{jk}\left(
t,x\right)  \in C^{\infty}\left(  \mathbb{R\times R}^{3}\right)  $ which
entries satisfy (\ref{t23}) for $\left\vert \alpha\right\vert \geq1\ $and
$\mathbf{V}_{2}\left(  t,x\right)  $ is a $\left(  4\times4\right)
-$matrix-valued function whose components verify (\ref{t23}) for $\left\vert
\alpha\right\vert \geq0.$ Then,
\begin{equation}
\left\Vert \psi\left(  t,\cdot\right)  \right\Vert _{M^{p,q}}\leq
C_{T}\left\Vert \psi_{0}\right\Vert _{M^{p,q}},\text{ }\psi_{0}\in
\mathcal{S}\left(  \mathbb{R}^{N}\right)  , \label{t56}%
\end{equation}
uniformly for $t\in\lbrack-T,T].$
\end{theorem}

\begin{rem}
Here are some comments on Theorems \ref{TDiracfree} and \ref{TDirac}.

\begin{enumerate}
\item For definiteness, we choose the space dimension to be $N=3.$ Theorems
\ref{TDiracfree} and \ref{TDirac} remain valid for the Dirac equation in any dimension.

\item We note that we prove Theorem \ref{TDirac} for more general equations
(\ref{eq1}). See Theorems \ref{T1} and \ref{T2}.

\item We expect that in general it is not possible to replace the $M^{p,p}%
$-norm in (\ref{t55}) by the $M^{p,q}$-norm$.$ In the case of the
Schr\"{o}dinger equations it is known that (\ref{t4}) might be false if $p\neq
q.$ Indeed, let $u\left(  t,x\right)  $ solves (\ref{Shrodinger}) with
$V\left(  t,x\right)  =\frac{1}{2}\left\vert x\right\vert ^{2}.$ Then (see
\cite{Kato1})
\[
\left\Vert u\left(  \frac{\pi}{2},\cdot\right)  \right\Vert _{M_{\phi\left(
\frac{\pi}{2},\cdot\right)  }^{p,q}}=\left\Vert \left\Vert W_{\phi_{0}}%
u_{0}\left(  \xi,x\right)  \right\Vert _{L_{x}^{p}}\right\Vert _{L_{\xi}^{q}%
}.
\]
However, $\left\Vert \left\Vert W_{\phi_{0}}u_{0}\left(  \xi,x\right)
\right\Vert _{L_{x}^{p}}\right\Vert _{L_{\xi}^{q}}\leq C\left\Vert \left\Vert
W_{\phi_{0}}u_{0}\left(  x,\xi\right)  \right\Vert _{L_{x}^{p}}\right\Vert
_{L_{\xi}^{q}}$ it is not true in general.

\item We observe that in Theorem \ref{TDirac} the potential $\mathbf{V}%
_{2}\left(  t,x\right)  $ is a general $\left(  4\times4\right)
-$matrix-valued function, not necessarily Hermitian.

\item In the second part of Theorem \ref{TDirac} the potential $\mathbf{V}%
_{1}\left(  t,x\right)  $ is a general Hermitian matrix. Therefore, it
includes the case of the electromagnetic potential%
\[
\ \mathbf{V}_{1}\left(  t,x\right)  =%
\begin{pmatrix}
Q_{+}\left(  t,x\right)  & \alpha\cdot A\left(  t,x\right) \\
\alpha\cdot A\left(  t,x\right)  & Q_{-}\left(  t,x\right)
\end{pmatrix}
,
\]
where $A\left(  t,x\right)  =\left(  A_{1}\left(  t,x\right)  ,A_{2}\left(
t,x\right)  ,A_{3}\left(  t,x\right)  \right)  ,$ and $Q_{\pm},A_{j}\in
C^{\infty}\left(  \mathbb{R\times R}^{3}\right)  ,$ $j=1,2,3,$ satisfy
(\ref{t23}) for $\left\vert \alpha\right\vert \geq1.$
\end{enumerate}
\end{rem}

\textbf{Comments on the proof. }Our proof is based on the strategy developed
in the papers \cite{Kato1}-\cite{Kato3}. As a first step, we need to derive a
representation for the Dirac operator that permits to transform the Dirac
equation (\ref{Dirac}) into a system of ordinary differential equations as the
transform obtained in \cite{Kato2} for the Schr\"{o}dinger equation. Since the
free Dirac equation is a system of coupled equations, it seems impossible to
obtain such a representation for the Dirac operator. Fortunately for us, it is
enough to obtain approximate representation that, to the main order, transform
the original equation (\ref{Dirac}) into a new one that can be solved
"approximately" (see (\ref{t1}), (\ref{t3}) and (\ref{t5}) below). Then, using
the integral equation (\ref{t5}) for the solution to this transformed
equation, we are able to obtain the desired estimates for the solution of
(\ref{Dirac}). Since we use an approximate representation, the window $\phi$
in (\ref{t55}) and (\ref{t56}) is fixed, while in (\ref{t4})\ and (\ref{t35})
$\phi\left(  t,x\right)  $ is asked to solve the free Schr\"{o}dinger
equation. We observe that in order to obtain the integral representation
(\ref{t5}), we restrict ourselves in (\ref{t55}) to diagonal quadratic
potentials $Q\left(  t,x\right)  I_{4}$.

The rest of the paper is organized as follows. In Section 2, we introduce some
notation, we recall some definitions and properties of the vector-valued
modulation spaces and we present some known results for the Dirac operator.
Section 3 is dedicated to the proof of our main results and it is divided in
two parts: in the first part we prove Theorem \ref{TDiracfree}, whilst in the
second part, we present the proof of Theorem \ref{TDirac}.

\section{Preliminaries.}

\subsection{Notation.}

Let $N,m\geq1$ be entire. For $1\leq p<\infty$ we denote by $L^{p}\left(
\mathbb{R}^{N};\mathbb{C}^{m}\right)  $ the Lebesgue spaces of $\mathbb{C}%
^{m}$-vector valued functions. Also, we introduce the weighted $L^{2}\left(
\mathbb{R}^{N};\mathbb{C}^{m}\right)  $ spaces for $s\in\mathbb{R},$
$L_{s}^{2}:=\{f:\left\langle x\right\rangle ^{s}f\left(  x\right)  \in
L^{2}\left(  \mathbb{R}^{N};\mathbb{C}^{m}\right)  \},$ $\left\Vert
f\right\Vert _{L_{s}^{2}\left(  \mathbb{R}^{N};\mathbb{C}^{m}\right)
}:=\left\Vert \left\langle x\right\rangle ^{s}f\left(  x\right)  \right\Vert
_{L^{2}\left(  \mathbb{R}^{N};\mathbb{C}^{m}\right)  },$ where $\left\langle
x\right\rangle =\left(  1+\left\vert x\right\vert ^{2}\right)  ^{1/2}.$ For
$1\leq p,q\leq\infty,$ $r,s\in\mathbb{R}$, $f\in L_{r,s}^{p,q}\left(
\mathbb{R}^{2N};\mathbb{C}^{m}\right)  $ if
\[
\left\Vert f\right\Vert _{L_{r,s}^{p,q}\left(  \mathbb{R}^{N};\mathbb{C}%
^{m}\right)  }=\left(
{\displaystyle\int\limits_{\mathbb{R}^{N}}}
\left(
{\displaystyle\int\limits_{\mathbb{R}^{N}}}
\left\vert f\left(  x,\xi\right)  \right\vert _{\mathbb{C}^{m}}^{p}%
\left\langle x\right\rangle ^{pr}dx\right)  ^{q/p}\left\langle \xi
\right\rangle ^{qs}d\xi\right)  ^{1/q}<\infty,
\]
with the standard modification when $p$ or $q$ are equal to infinity. We
denote by $\left(  \cdot,\cdot\right)  $ the $L^{2}$ scalar product. The
Fourier transform $\mathcal{F}$ is given by
\[
\hat{f}\left(  \xi\right)  =\left(  \mathcal{F}f\right)  \left(  \xi\right)
:=\left(  2\pi\right)  ^{-N/2}%
{\displaystyle\int\limits_{\mathbb{R}^{N}}}
e^{-i\xi\cdot x}f\left(  x\right)  dx
\]
($\xi\cdot x=\sum_{j=1}^{N}x_{j}\xi_{j}$) and the inverse Fourier transform
$\mathcal{F}^{-1}$ is defined by%
\[
\check{f}\left(  \xi\right)  =\left(  \mathcal{F}^{-1}f\right)  \left(
x\right)  :=\left(  2\pi\right)  ^{-N/2}%
{\displaystyle\int\limits_{\mathbb{R}^{N}}}
e^{i\xi\cdot x}f\left(  \xi\right)  d\xi.
\]
Finally, we denote by $C>0$ constants that may be different in each occasion.

\subsection{The wave packet transform of vector-valued tempered
distributions.}

We now present the definition of the wave packet transform of short-time
Fourier transform and recall some results of \cite{Wahlberg}. The wave packet
transform of $f\in L^{p}\left(  \mathbb{R}^{N};\mathbb{C}^{m}\right)  $ with
respect to a window function $\phi\in L^{p^{\prime}}\left(  \mathbb{R}%
^{N}\right)  $ is defined by%
\begin{equation}
W_{\phi}f\left(  x,\xi\right)  :=\int_{\mathbb{R}^{N}}\overline{\phi\left(
x-y\right)  }f\left(  y\right)  e^{-iy\cdot\xi}dy. \label{t57}%
\end{equation}
Similarly to the scalar case, the wave packet transform for a vector valued
distribution $f\in\mathcal{S}^{\prime}\left(  \mathbb{R}^{N};\mathbb{C}%
^{m}\right)  $ with respect to the window $\phi\in\mathcal{S}\left(
\mathbb{R}^{N}\right)  $ is defined by the right-hand side of (\ref{t57}),
where the "integral" in this case means distribution action. The following
result holds (see Lemma 2.1 of \cite{Wahlberg}).

\begin{proposition}
Let $f\in\mathcal{S}^{\prime}\left(  \mathbb{R}^{N};\mathbb{C}^{m}\right)  $
and $\phi\in\mathcal{S}\left(  \mathbb{R}^{N}\right)  .$ Then, $W_{\phi}f\in
C^{\infty}\left(  \mathbb{R}^{2N};\mathbb{C}^{m}\right)  $ and there exist an
integer $M>0$ and $C>0$ such that%
\[
\left\vert W_{\phi}f\left(  x,\xi\right)  \right\vert _{\mathbb{C}^{m}}\leq
C\left(  1+\left\vert x\right\vert +\left\vert \xi\right\vert \right)
^{M},\text{ \ }\left(  x,\xi\right)  \in\mathbb{R}^{2N}.
\]

\end{proposition}

For a strongly measurable function $F:\mathbb{R}^{2N}\rightarrow\mathbb{C}%
^{m}$ and $\psi\in\mathcal{S}\left(  \mathbb{R}^{N}\right)  $ we define the
map%
\[
\left\langle W_{\psi}^{\ast}F,\phi\right\rangle :=%
{\displaystyle\iint\limits_{\mathbb{R}^{2N}}}
F\left(  y,\xi\right)  \left\langle \psi\left(  x-y\right)  ,\phi\left(
y\right)  \right\rangle e^{ix\cdot\xi}dy\text{\textit{\dj }}\xi,
\]
with \textit{\dj }$\xi=\left(  2\pi\right)  ^{-n}d\xi,$ and denote,
\begin{equation}
W_{\psi}^{\ast}F=%
{\displaystyle\iint\limits_{\mathbb{R}^{2N}}}
F\left(  y,\xi\right)  \psi\left(  x-y\right)  e^{ix\cdot\xi}%
dy\text{\textit{\dj }}\xi. \label{t6}%
\end{equation}
The next result shows that the wave packet transform is invertible on
$\mathcal{S}^{\prime}\left(  \mathbb{R}^{N};\mathbb{C}^{m}\right)  $ and the
adjoint operator $W_{\gamma}^{\ast}$ is defined by (\ref{t6}) (see Proposition
2.5 of \cite{Wahlberg}).

\begin{proposition}
\bigskip Let $\phi,\psi\in\mathcal{S}\left(  \mathbb{R}^{N}\right)  $ be such
that $\left\langle \psi,\phi\right\rangle \neq0.$ Then,
\[
f=\left\langle \psi,\phi\right\rangle ^{-1}W_{\psi}^{\ast}W_{\phi}f,\text{
}f\in\mathcal{S}^{\prime}.
\]

\end{proposition}

\subsection{Modulation spaces.\label{S2.3}}

Let us now define the modulation spaces and recall some properties of these
spaces. Let $1\leq p,q\leq\infty,$ $r,s\in\mathbb{R}$. For $f\in
\mathcal{S}^{\prime}\left(  \mathbb{R}^{N};\mathbb{C}^{m}\right)  $ and
$W_{\phi}f\in L_{r,s}^{p,q}\left(  \mathbb{R}^{2N};\mathbb{C}^{m}\right)  $,
with $\phi\in\mathcal{S}\left(  \mathbb{R}^{N}\right)  \diagdown\{0\},$ the
modulation space norm is defined by (see \cite{Wahlberg}, \cite{Toft})%
\[
\left\Vert f\right\Vert _{M_{r,s}^{p,q}\left(  \mathbb{R}^{N};\mathbb{C}%
^{m}\right)  }:=\left\Vert W_{\phi}f\right\Vert _{L_{r,s}^{p,q}\left(
\mathbb{R}^{N};\mathbb{C}^{m}\right)  },\text{ }\phi\in\mathcal{S}\left(
\mathbb{R}^{N}\right)  \diagdown\{0\}.
\]
At first, this definition is dependent on the window $\phi\in\mathcal{S}%
\left(  \mathbb{R}^{N}\right)  \diagdown\{0\}.$ Nevertheless, it is not: it
follows from Proposition 3.2 of \cite{Wahlberg} that if $\phi,\psi
\in\mathcal{S}\left(  \mathbb{R}^{N}\right)  \diagdown\{0\},$ then the norm
$\left\Vert \cdot\right\Vert _{\left(  M_{r,s}^{p,q}\right)  _{\phi}\left(
\mathbb{R}^{N};\mathbb{C}^{m}\right)  }$ associated to $\phi$ and the norm
$\left\Vert \cdot\right\Vert _{\left(  M_{r,s}^{p,q}\right)  _{\psi}\left(
\mathbb{R}^{N};\mathbb{C}^{m}\right)  }$ corresponding to $\psi$ are
equivalent. For $1\leq p,q\leq\infty,$ $r,s\in\mathbb{R}$, the modulation
space $M_{r,s}^{p,q}\left(  \mathbb{R}^{N};\mathbb{C}^{m}\right)
\subset\mathcal{S}^{\prime}\left(  \mathbb{R}^{N};\mathbb{C}^{m}\right)  ,$ is
defined as the set of all $f\in\mathcal{S}^{\prime}\left(  \mathbb{R}%
^{N};\mathbb{C}^{m}\right)  $ such that $\left\Vert f\right\Vert
_{M_{r,s}^{p,q}\left(  \mathbb{R}^{N};\mathbb{C}^{m}\right)  }<\infty.$ We
write $M^{p,q}\left(  \mathbb{R}^{N};\mathbb{C}^{m}\right)  :=M_{0,0}%
^{p,q}\left(  \mathbb{R}^{N};\mathbb{C}^{m}\right)  $ \ and $M^{p,p}\left(
\mathbb{R}^{N};\mathbb{C}^{m}\right)  :=M^{p}\left(  \mathbb{R}^{N}%
;\mathbb{C}^{m}\right)  .$ The following result holds (see Proposition 3.3 of
\cite{Wahlberg}):

\begin{proposition}
For any $1\leq p,q<\infty,$ and $r,s\in\mathbb{R}$, the set $\mathcal{S}%
\left(  \mathbb{R}^{N};\mathbb{C}^{m}\right)  $ is dense in $M_{r,s}%
^{p,q}\left(  \mathbb{R}^{N};\mathbb{C}^{m}\right)  .$ If $p_{1}\leq p_{2},$
$q_{1}\leq q_{2},$ $r_{1}\geq r_{2},$ $s_{1}\geq s_{2},$ $M_{r_{1},s_{1}%
}^{p_{1},q_{1}}\left(  \mathbb{R}^{N};\mathbb{C}^{m}\right)  \hookrightarrow
M_{r_{2},s_{2}}^{p_{2},q_{2}}\left(  \mathbb{R}^{N};\mathbb{C}^{m}\right)  .$
Moreover, $M_{r,0}^{2}\left(  \mathbb{R}^{N};\mathbb{C}^{m}\right)  =L_{r}%
^{2}\left(  \mathbb{R}^{N};\mathbb{C}^{m}\right)  $ and $M_{0,s}^{2}\left(
\mathbb{R}^{N};\mathbb{C}^{m}\right)  =\mathcal{F}L_{s}^{2}\left(
\mathbb{R}^{N};\mathbb{C}^{m}\right)  $ with equivalent norms.
\end{proposition}

As in the case of the modulation spaces for scalar distributions (see
\cite{Feichtinger}), the complex interpolation theory for the modulation
spaces $M^{p,q}\left(  \mathbb{R}^{N};\mathbb{C}^{m}\right)  $ stands as follows:

\begin{proposition}
\label{PropInterpolation}Let $0<\theta<1$ and $1\leq p_{i},q_{i}\leq\infty,$
$i=1,2.$ Set $1/p=\left(  1-\theta\right)  /p_{1}+\theta/p_{2},$ $1/q=\left(
1-\theta\right)  /q_{1}+\theta/q_{2},$ then $(M^{p_{1},q_{1}},M^{p_{2},q_{2}%
})_{\left[  \theta\right]  }=M^{p,q}.$
\end{proposition}

We conclude this Subsection by presenting the following result (see Corollary
2.3 of \cite{Toft}).

\begin{proposition}
\label{LemmaToft}For $s\in\mathbb{R}$, $\left\langle D\right\rangle ^{s}$ is a
continuous bijective map from $M_{0,s}^{p,q}$ to $M^{p,q}$, with continuous inverse.
\end{proposition}

\subsection{Dirac equation.}

The free Dirac operator $H_{0}$ defined by (\ref{D1}) is a self-adjoint
operator on $L^{2}\left(  \mathbb{R}^{3};\mathbb{C}^{4}\right)  $ with domain
$D\left(  H_{0}\right)  =\mathcal{H}^{1}\left(  \mathbb{R}^{3};\mathbb{C}%
^{4}\right)  ,$ the Sobolev space of order $1$ (\cite{Adams}). We can
diagonalize $H_{0}$ by the Fourier transform $\mathcal{F}$. Actually,
$\mathcal{F}H_{0}\mathcal{F}^{\ast}$ acts as multiplication by the matrix
$h_{0}\left(  \xi\right)  =\alpha\cdot\xi+m\beta.$ This matrix has two
eigenvalues $E=\pm\sqrt{\xi^{2}+m^{2}}$ and each eigenspace $X^{\pm}\left(
\xi\right)  $ is a two-dimensional subspace of $\mathbb{C}^{4}.$ The
orthogonal projections onto these eigenspaces are given by (see Section 2 of
\cite{NW})%
\begin{equation}
P_{\pm}\left(  \xi\right)  :=\frac{1}{2}\left(  I_{4}\pm\left(  \xi^{2}%
+m^{2}\right)  ^{-1/2}\left(  \alpha\cdot\xi+m\beta\right)  \right)  .
\label{basicnotions49}%
\end{equation}
Note that
\[
P_{\pm}\left(  \xi\right)  \mathcal{F}=\left(  \mathcal{F}\mathbf{P}_{\pm
}\right)  \left(  \xi\right)  ,
\]
where%
\begin{equation}
\mathbf{P}_{\pm}:=\frac{1}{2}\left(  I_{4}\pm\frac{H_{0}}{\left\vert
H_{0}\right\vert }\right)  , \label{t48}%
\end{equation}
are the projections on positive and negative energies of the Dirac operator
$H_{0}.$

\section{Proof of Theorems \ref{TDiracfree} and \ref{TDirac}.}

\subsection{Proof of Theorem \ref{TDiracfree}.}

Estimate (\ref{t54}) follows from (\ref{t56}) for $\mathbf{V}\left(
t,x\right)  \equiv0$. In order to prove the second assertion, we multiply the
both sides of equation (\ref{Dirac}) with $\mathbf{V}\left(  t,x\right)
\equiv0$ by $\mathbf{P}_{\pm}$, defined in (\ref{t48}). Then, we obtain two
equations%
\begin{equation}
\left\{
\begin{array}
[c]{c}%
i\partial_{t}\psi_{\pm}\left(  t,x\right)  =\pm\left(  \sqrt{m^{2}-\Delta
}\right)  \psi_{\pm}\left(  t,x\right)  ,\text{ \ }(t,x)\in\mathbb{R\times
R}^{N},\\
\psi_{\pm}\left(  0,x\right)  =\left(  \psi_{0}\right)  _{\pm}\left(
x\right)  ,\text{ }x\in\mathbb{R}^{N},
\end{array}
\right.  \label{t50}%
\end{equation}
where $\psi_{\pm}\left(  t,x\right)  :=\mathbf{P}_{\pm}\psi\left(  t,x\right)
$. We need an estimate for the Klein-Gordon semigroup $e^{it\omega^{1/2}},$
$\omega=\left(  m^{2}-\Delta\right)  $ in terms of the modulation spaces. For
$2\leq p<\infty,$ $0<q<\infty$, $\phi\in\mathcal{S}\left(  \mathbb{R}%
^{N}\right)  \diagdown\{0\}$ and $f\in M^{p,q},$ we have the following
inequality (see Proposition 4.2 of \cite{Wang})
\begin{equation}
\left\Vert e^{it\omega^{1/2}}f\right\Vert _{M_{0,-2\sigma}^{p,q}}\leq
C\left\langle t\right\rangle ^{-N\theta\lbrack1/2-1/p]}\left\Vert f\right\Vert
_{M^{p^{\prime},q}},\text{ \ }\frac{1}{p}+\text{\ }\frac{1}{p^{\prime}}=1,
\label{t49}%
\end{equation}
where $\theta\in\lbrack0,1]$ and $2\sigma=\left(  N+2\right)  \theta\left(
\frac{1}{2}-\frac{1}{p}\right)  .$ Using (\ref{t49}) in (\ref{t50}) we
estimate
\begin{equation}
\left\Vert \psi_{\pm}\left(  t,\cdot\right)  \right\Vert _{M_{0,-2\sigma
}^{p,q}}\leq C\left\langle t\right\rangle ^{-N\theta\lbrack1/2-1/p]}\left\Vert
\left(  \psi_{0}\right)  _{\pm}\right\Vert _{M^{p^{\prime},q}}. \label{t51}%
\end{equation}
Since $\mathbf{P}_{+}+\mathbf{P}_{-}=I,$ we have%
\begin{equation}
\left\Vert \psi\left(  t,\cdot\right)  \right\Vert _{M_{0,-2\sigma}^{p,q}}%
\leq\left\Vert \psi_{+}\left(  t,\cdot\right)  \right\Vert _{M_{0,-2\sigma
}^{p,q}}+\left\Vert \psi_{-}\left(  t,\cdot\right)  \right\Vert
_{M_{0,-2\sigma}^{p,q}}. \label{t58}%
\end{equation}
Moreover, as $\alpha_{j}^{2}=I,$ $j=1,2,3,4,$
\[
\left\Vert H_{0}\psi\left(  t,\cdot\right)  \right\Vert _{M_{0,-2\sigma}%
^{p,q}}\leq\left\Vert \psi\left(  t,\cdot\right)  \right\Vert _{M_{0,-2\sigma
+1}^{p,q}}+m\left\Vert \psi\left(  t,\cdot\right)  \right\Vert _{M_{0,-2\sigma
}^{p,q}}\leq\left(  1+m\right)  \left\Vert \psi\left(  t,\cdot\right)
\right\Vert _{M_{0,-2\sigma+1}^{p,q}}.
\]
Then, using Lemma \ref{LemmaToft} we deduce%
\begin{equation}
\left\Vert \left(  \psi_{0}\right)  _{\pm}\right\Vert _{M_{0,-2\sigma}^{p,q}%
}=\left\Vert \mathbf{P}_{\pm}\psi_{0}\right\Vert _{M_{0,-2\sigma}^{p,q}}%
\leq\frac{1}{2}\left(  \left\Vert \psi_{0}\right\Vert _{M_{0,-2\sigma}^{p,q}%
}+\left\Vert \frac{H_{0}}{\left\langle D\right\rangle }\psi_{0}\right\Vert
_{M_{0,-2\sigma}^{p,q}}\right)  \leq\left(  1+\frac{m}{2}\right)  \left\Vert
\psi_{0}\right\Vert _{M_{0,-2\sigma}^{p,q}}. \label{t52}%
\end{equation}
Therefore, by (\ref{t51}), (\ref{t58}) and (\ref{t52}) we attain (\ref{t53}).
Theorem \ref{TDiracfree} is proved.

\subsection{Proof of Theorem \ref{TDirac}.}

Instead of proving Theorem \ref{TDirac} for equation (\ref{Dirac}) directly,
we consider a more general system%

\begin{equation}
\left\{
\begin{array}
[c]{c}%
i\partial_{t}u\left(  t,x\right)  =\mathbf{a}\left(  D\right)  u\left(
t,x\right)  +\mathbf{V}\left(  t,x\right)  u\left(  t,x\right)  ,\text{
\ }(t,x)\in\mathbb{R\times R}^{N},\\
u\left(  0,x\right)  =u_{0}\left(  x\right)  ,\text{ }x\in\mathbb{R}^{N},
\end{array}
\right.  \label{eq1}%
\end{equation}
where $\mathbf{a}\left(  D\right)  =\mathcal{F}^{-1}a\left(  \xi\right)
\mathcal{F}^{-1},$ $a\left(  \xi\right)  $ is an Hermitian $\left(  m\times
m\right)  $-matrix valued symbol such that
\begin{equation}
b\left(  \eta;\xi\right)  =a\left(  \xi-\eta\right)  -a\left(  \xi\right)
\label{t45}%
\end{equation}
satisfies%
\begin{equation}
\left\vert \partial_{\eta}^{\alpha}b\left(  \eta;\xi\right)  \right\vert \leq
C_{\alpha}\left\langle \eta\right\rangle ^{k},\text{ \ for all }\xi,\eta
\in\mathbb{R}^{N}, \label{t46}%
\end{equation}
for some $k\geq0,$ and $\mathbf{V}\left(  t,x\right)  $ is a $\left(  m\times
m\right)  -$matrix valued function. Since $\mathcal{F}^{-1}H_{0}%
\mathcal{F}=\alpha\cdot\xi+m\beta,$ the case of the equation (\ref{Dirac}) is
included in (\ref{eq1}). We prove the following elementary estimate that is
involved in the proof of the main results.

\begin{lemma}
\label{L2}Suppose that $b$ satisfies (\ref{t46}). Then, for any $\phi
\in\mathcal{S}\left(  \mathbb{R}^{N}\right)  $ the estimate%
\begin{equation}
\left\vert \partial_{z}^{\beta}\left(  \mathcal{F}_{\eta\rightarrow z}%
^{-1}\left(  \left(  b\left(  \eta,\xi\right)  \right)  \phi\left(
\eta\right)  \right)  \right)  \left(  z\right)  \right\vert \leq C_{\beta
}\left\langle z\right\rangle ^{-2n}, \label{t14}%
\end{equation}
for any $\left\vert \beta\right\vert \geq0$ and $n\in\mathbb{N}$ such that
$2n>N.$
\end{lemma}

\begin{proof}
Note that
\begin{equation}
\partial_{z}^{\beta}\left(  \mathcal{F}_{\eta\rightarrow z}^{-1}\left(
b\left(  \eta,\xi\right)  \phi\left(  \eta\right)  \right)  \right)  \left(
z\right)  =\int_{\mathbb{R}^{N}}e^{iz\cdot\eta}\left(  i\eta\right)  ^{\beta
}b\left(  \eta,\xi\right)  \phi\left(  \eta\right)  d\eta. \label{t15}%
\end{equation}
Using the equality%
\begin{equation}
\left(  1-\Delta_{\eta}\right)  ^{n}e^{iz\cdot\eta}=\left\langle
z\right\rangle ^{2n}e^{iz\cdot\eta} \label{t11}%
\end{equation}
in the right-hand side of (\ref{t15}) and integrating by parts we obtain the
estimate (\ref{t14}).
\end{proof}

The first assertion of Theorem \ref{TDirac} is consequence of the following result.

\begin{theorem}
\bigskip\label{T1}Let $1\leq p\leq\infty$ and $T>0.$ Set $m\in\mathbb{N}$. Let
$a\left(  \xi\right)  =\{a_{jk}\left(  \xi\right)  \}$ be an Hermitian
$\left(  m\times m\right)  $-matrix valued function such that $a_{jk}\left(
\xi\right)  \in C^{\infty}\left(  \mathbb{R}^{N}\right)  $ and that satisfies
the estimate (\ref{t46}). Suppose that $\mathbf{V}\left(  t,x\right)
=Q\left(  t,x\right)  I_{m}+\mathbf{V}_{2}\left(  t,x\right)  ,$ where
$Q\left(  t,x\right)  \in C^{\infty}\left(  \mathbb{R}^{N}\times\mathbb{R}%
^{N}\right)  $ is a real-valued function satisfying (\ref{t23}) for
$\left\vert \alpha\right\vert \geq2$ and $\mathbf{V}_{2}\left(  t,x\right)
\in C^{\infty}\left(  \mathbb{R}^{N}\times\mathbb{R}^{N}\right)  $ is a
$\left(  m\times m\right)  -$matrix-valued function that verifies (\ref{t23})
for $\left\vert \alpha\right\vert \geq0.$ Then, the solution $u\left(
t,x\right)  $ of (\ref{eq1})
\begin{equation}
\left\Vert u\left(  t,\cdot\right)  \right\Vert _{M^{p}}\leq C_{T}\left\Vert
u_{0}\right\Vert _{M^{p}}, \label{t20}%
\end{equation}
uniformly for $t\in\lbrack-T,T].$
\end{theorem}

\begin{proof}
We consider the case $t\in\left[  0,T\right]  .$ Note that
\begin{equation}
\left.
\begin{array}
[c]{c}%
W_{\phi}\left(  a\left(  D\right)  u\right)  \left(  x,\xi\right)  =%
{\displaystyle\int\limits_{\mathbb{R}^{N}}}
e^{-iy\cdot\xi}\overline{\phi\left(  x-y\right)  }\left(
{\displaystyle\int\limits_{\mathbb{R}^{N}}}
e^{i\eta\cdot y}a\left(  \eta\right)  \hat{u}\left(  \eta\right)
d\eta\right)  dy\\
=%
{\displaystyle\int\limits_{\mathbb{R}^{N}}}
e^{-iy\cdot\xi}\overline{\phi\left(  x-y\right)  }\left(
{\displaystyle\int\limits_{\mathbb{R}^{N}}}
e^{i\left(  \xi-\eta\right)  \cdot y}a\left(  \xi-\eta\right)  \hat{u}\left(
\xi-\eta\right)  d\eta\right)  dy\\
=a\left(  \xi\right)  W_{\phi}u\left(  x,\xi\right)  +R_{0}u\left(
x,\xi\right)  ,
\end{array}
\right.  \label{t1}%
\end{equation}
where%
\begin{equation}
R_{0}u\left(  x,\xi\right)  =%
{\displaystyle\int\limits_{\mathbb{R}^{N}}}
S\left(  x-y,\xi\right)  e^{-i\xi\cdot y}u\left(  y\right)  dy \label{t10}%
\end{equation}
and
\[
S\left(  z,\xi\right)  :=\left(  \mathcal{F}_{\eta\rightarrow z}^{-1}\left(
\left(  a\left(  \xi-\eta\right)  -a\left(  \xi\right)  \right)  \left(
\mathcal{F}\overline{\phi}\right)  \left(  \eta\right)  \right)  \right)
\left(  z\right)  .
\]
Expanding the potential $Q$ in Taylor's series, we have%
\[
Q\left(  t,y\right)  =Q\left(  t,x\right)  +\left(  y-x\right)  \cdot\left(
\nabla Q\right)  \left(  t,x\right)  +\sum_{j,k=1}^{N}\left(  y_{j}%
-x_{j}\right)  \left(  y_{k}-x_{k}\right)  Q_{jk}\left(  t,y,x\right)  ,
\]
with%
\[
Q_{jk}\left(  t,y,x\right)  :=%
{\displaystyle\int\limits_{0}^{1}}
\left(  \partial_{x_{j}}\partial_{x_{k}}Q\left(  t,x+\theta\left(  y-x\right)
\right)  \right)  \left(  1-\theta\right)  d\theta.
\]
Then,
\begin{equation}
\left.
\begin{array}
[c]{c}%
W_{\phi}\left(  Qu\right)  \left(  t,x,\xi\right)  =\left(  Q\left(
t,x\right)  +i\left(  \nabla Q\left(  t,x\right)  \right)  \cdot\nabla_{\xi
}-\left(  x\cdot\nabla Q\right)  \left(  t,x\right)  \right)  W_{\phi}u\left(
t,x,\xi\right)  +Ru\left(  t,x,\xi\right)  ,
\end{array}
\right.  \label{t2}%
\end{equation}
where%
\begin{equation}
Ru\left(  t,x,\xi\right)  =%
{\displaystyle\int\limits_{\mathbb{R}^{N}}}
e^{-iy\cdot\xi}\overline{\phi\left(  x-y\right)  }%
{\displaystyle\sum\limits_{j,k=1}^{N}}
\left(  y_{j}-x_{j}\right)  \left(  y_{k}-x_{k}\right)  Q_{jk}\left(
y,x\right)  u\left(  y\right)  dy. \label{t7}%
\end{equation}
Thus, by (\ref{t1}) and (\ref{t2}) we transform equation (\ref{eq1}) into
\begin{equation}
\left\{
\begin{array}
[c]{c}%
\left(  i\partial_{t}-i\left(  \nabla Q\left(  t,x\right)  \right)
\cdot\nabla_{\xi}-Q\left(  t,x\right)  +\left(  x\cdot\nabla Q\right)  \left(
t,x\right)  -a\left(  \xi\right)  \right)  W_{\phi}u\left(  t,x,\xi\right) \\
=Ru\left(  x,\xi\right)  +\tilde{R}u\left(  x,\xi\right)  +R_{0}u\left(
x,\xi\right)  ,\\
W_{\phi}\left(  u\right)  \left(  0,x,\xi\right)  =W_{\phi}\left(
u_{0}\right)  \left(  x,\xi\right)  .
\end{array}
\right.  \label{t3}%
\end{equation}
with%
\[
\tilde{R}u\left(  x,\xi\right)  :=%
{\displaystyle\int\limits_{\mathbb{R}^{N}}}
e^{-iy\cdot\xi}\overline{\phi\left(  x-y\right)  }\mathbf{V}_{2}\left(
t,y\right)  u\left(  y\right)  dy
\]
We solve problem (\ref{t3}) by the method of characteristics. The solution to
(\ref{t3}) is given by
\begin{equation}
W_{\phi}u\left(  t,x,\xi\right)  =e^{-i%
{\textstyle\int_{0}^{t}}
h\left(  s;t,x,\xi\right)  ds}\left(  W_{\phi}\left(  u_{0}\right)  \left(
x,g\left(  0;t,x,\xi\right)  \right)  -i%
{\displaystyle\int\limits_{0}^{t}}
e^{i%
{\textstyle\int_{0}^{\tau}}
h\left(  s;t,x,\xi\right)  ds}\left(  Ru+\tilde{R}u+R_{0}u\right)  \left(
\tau,x,g\left(  \tau;t,x,\xi\right)  \right)  d\tau\right)  , \label{t5}%
\end{equation}
where%
\[
h\left(  s;t,x,\xi\right)  :=a\left(  g\left(  s;t,x,\xi\right)  \right)
+Q\left(  s,x\right)  -\left(  x\cdot\nabla Q\right)  \left(  s,x\right)
\]
and $g\left(  s;t,x,\xi\right)  =\xi-\int_{t}^{s}\left(  \nabla\mathbf{V}%
\right)  \left(  \tau,x\right)  d\tau$. Taking the $L^{p}$-norm with respect
to $x$ and $\xi$ on the both side of (\ref{t5}) we obtain%
\begin{equation}
\left\Vert u\left(  t,\cdot\right)  \right\Vert _{M^{p}}=\left\Vert W_{\phi
}u\left(  t,x,\xi\right)  \right\Vert _{L_{x,\xi}^{p}}\leq\left\Vert
I_{1}\right\Vert _{L_{x,\xi}^{p}}+%
{\displaystyle\int\limits_{0}^{t}}
\left(  \left\Vert I_{2}\right\Vert _{L_{x,\xi}^{p}}+\left\Vert I_{3}%
\right\Vert _{L_{x,\xi}^{p}}\right)  d\tau, \label{t16}%
\end{equation}
where%
\begin{equation}
I_{1}:=W_{\phi}\left(  u_{0}\right)  \left(  x,g\left(  0;t,x,\xi\right)
\right)  , \label{t26}%
\end{equation}%
\[
I_{2}:=Ru\left(  \tau,x,g\left(  \tau;t,x,\xi\right)  \right)
\]%
\[
I_{3}:=\tilde{R}u\left(  \tau,x,g\left(  \tau;t,x,\xi\right)  \right)
\]
and%
\[
I_{4}:=R_{0}u\left(  \tau,x,g\left(  \tau;t,x,\xi\right)  \right)  .
\]
\ 

We begin by estimating $I_{1}.$ Let us consider the change of variables
$\Xi=g\left(  0;t,x,\xi\right)  .$ Since
\[
\det\mathbf{J}\left(  g\left(  s;t,x,\xi\right)  \right)  =1,\text{ for all
}s,t,x\text{ and }\xi,
\]
($\mathbf{J}$ denotes the Jacobian matrix) the implicit function theorem imply%
\begin{equation}
\left\vert \frac{\partial\left(  \xi\right)  }{\partial\left(  \Xi\right)
}\right\vert =1. \label{t47}%
\end{equation}
Thus,
\begin{equation}
\left\Vert I_{1}\right\Vert _{L_{x,\xi}^{p}}=\left(
{\displaystyle\iint\limits_{\mathbb{R}^{2N}}}
\left\vert W_{\phi}\left(  u_{0}\right)  \left(  x,\Xi\right)  \right\vert
^{p}\left\vert \frac{\partial\left(  \xi\right)  }{\partial\left(  \Xi\right)
}\right\vert dxd\Xi\right)  ^{1/p}=\left\Vert u_{0}\right\Vert _{M^{p}}.
\label{t17}%
\end{equation}

Next, we consider $I_{2}.$ By the inversion formula (\ref{t6}), from
(\ref{t7}) we deduce%
\begin{equation}
Ru\left(  t,x,\xi\right)  =\frac{1}{\left\Vert \phi\right\Vert _{L^{2}}^{2}}%
{\displaystyle\sum\limits_{j,k=1}^{N}}
{\displaystyle\iiint\limits_{\mathbb{R}^{3N}}}
\phi_{jk}\left(  y-x\right)  Q_{jk}\left(  t,y,x\right)  \phi\left(
y-z\right)  W_{\phi}u\left(  t,z,\eta\right)  e^{iy\cdot\left(  \eta
-\xi\right)  }dz\text{\textit{\dj }}\eta dy, \label{t29}%
\end{equation}
where $\phi_{jk}\left(  x\right)  :=x_{j}x_{k}\overline{\phi\left(  x\right)
}.$ Then,
\begin{equation}
\left.
\begin{array}
[c]{c}%
I_{2}=Ru\left(  \tau,x,g\left(  \tau;t,x,\xi\right)  \right)  =\dfrac
{1}{\left\Vert \phi\right\Vert _{L^{2}}^{2}}%
{\displaystyle\sum\limits_{j,k=1}^{N}}
{\displaystyle\iiint\limits_{\mathbb{R}^{3N}}}
\phi_{jk}\left(  y-x\right) \\
\times\phi\left(  y-z\right)  Q_{jk}\left(  \tau,y,x\right)  W_{\phi}u\left(
\tau,z,\eta\right)  e^{iy\cdot\left(  \eta-g\left(  \tau;t,x,\xi\right)
\right)  }dz\text{\textit{\dj }}\eta dy.
\end{array}
\right.  \label{t8}%
\end{equation}
Let us consider $n\in\mathbb{N}$ such that $2n>N.$ Using (\ref{t11}) in the
right-hand side of (\ref{t8}) and integrating by parts we have%
\[
\left.
\begin{array}
[c]{c}%
\left\Vert I_{2}\right\Vert _{L_{x,\xi}^{p}}\\
\leq\dfrac{1}{\left\Vert \phi\right\Vert _{L^{2}}^{2}}%
{\displaystyle\sum\limits_{j,k=1}^{N}}
\left\Vert
{\displaystyle\iiint\limits_{\mathbb{R}^{3N}}}
\left\vert \left(  1-\Delta_{y}\right)  ^{n}\left(  \phi_{jk}\left(
y-x\right)  Q_{jk}\left(  \tau,y,x\right)  \phi\left(  y-z\right)  \right)
\right\vert \dfrac{\left\vert W_{\phi}u\left(  \tau,z,\eta\right)  \right\vert
}{\left\langle \eta-g\left(  \tau;t,x,\xi\right)  \right\rangle ^{2n}%
}dz\text{\textit{\dj }}\eta dy\right\Vert _{L_{x,\xi}^{p}}\\
\leq\dfrac{1}{\left\Vert \phi\right\Vert _{L^{2}}^{2}}%
{\displaystyle\sum\limits_{j,k=1}^{N}}
{\displaystyle\sum\limits_{\left\vert \beta_{1}\right\vert +\left\vert
\beta_{2}\right\vert +\left\vert \beta_{3}\right\vert \leq2n}}
\left\Vert
{\displaystyle\iiint\limits_{\mathbb{R}^{3N}}}
\left\vert \partial_{y}^{\beta_{1}}\phi_{jk}\left(  y-x\right)  \partial
_{y}^{\beta_{2}}Q_{jk}\left(  \tau,y,x\right)  \partial_{y}^{\beta_{3}}%
\phi\left(  y-z\right)  \right\vert \dfrac{\left\vert W_{\phi}u\left(
\tau,z,\eta\right)  \right\vert }{\left\langle \eta-g\left(  \tau
;t,x,\xi\right)  \right\rangle ^{2n}}dz\text{\textit{\dj }}\eta dy\right\Vert
_{L_{x,\xi}^{p}}.
\end{array}
\right.
\]
Since $\left\vert \partial_{y}^{\beta_{2}}Q_{jk}\left(  \tau,y,x\right)
\right\vert \leq C_{\beta_{2}},$ $C_{\beta_{2}}>0,$ we estimate%
\begin{equation}
\left\Vert I_{2}\right\Vert _{L_{x,\xi}^{p}}\leq\dfrac{1}{\left\Vert
\phi\right\Vert _{L^{2}}^{2}}%
{\displaystyle\sum\limits_{j,k=1}^{N}}
{\displaystyle\sum\limits_{\left\vert \beta_{1}\right\vert +\left\vert
\beta_{2}\right\vert +\left\vert \beta_{3}\right\vert \leq2n}}
\left\Vert
{\displaystyle\iiint\limits_{\mathbb{R}^{3N}}}
C_{\beta_{2}}\left\vert \partial_{y}^{\beta_{1}}\phi_{jk}\left(  y-x\right)
\right\vert \left\vert \partial_{y}^{\beta_{3}}\phi\left(  y-z\right)
\right\vert \dfrac{\left\vert W_{\phi}u\left(  \tau,z,\eta\right)  \right\vert
}{\left\langle \eta-g\left(  \tau;t,x,\xi\right)  \right\rangle ^{2n}%
}dz\text{\textit{\dj }}\eta dy\right\Vert _{L_{x,\xi}^{p}}. \label{t9}%
\end{equation}
Making the change of variables $\Xi=g\left(  \tau;t,x,\xi\right)  \ $in the
integral on the right-hand side of (\ref{t9}) and using (\ref{t47}) we obtain%
\begin{equation}
\left.
\begin{array}
[c]{c}%
\left\Vert I_{2}\right\Vert _{L_{x,\xi}^{p}}\\
\leq\dfrac{1}{\left\Vert \phi\right\Vert _{L^{2}}^{2}}%
{\displaystyle\sum\limits_{j,k=1}^{N}}
{\displaystyle\sum\limits_{\left\vert \beta_{1}\right\vert +\left\vert
\beta_{2}\right\vert +\left\vert \beta_{3}\right\vert \leq2n}}
\left(
{\displaystyle\iint\limits_{\mathbb{R}^{2N}}}
\left(
{\displaystyle\iiint\limits_{\mathbb{R}^{3N}}}
C_{\beta_{2}}\left\vert \partial_{y}^{\beta_{1}}\phi_{jk}\left(  y-x\right)
\right\vert \left\vert \partial_{y}^{\beta_{3}}\phi\left(  y-z\right)
\right\vert \dfrac{\left\vert W_{\phi}u\left(  \tau,z,\eta\right)  \right\vert
}{\left\langle \eta-\Xi\right\rangle ^{2n}}dz\text{\textit{\dj }}\eta
dy\right)  ^{p}dxd\Xi\right)  ^{1/p}.
\end{array}
\right.  \label{t13}%
\end{equation}
Then, by Young's inequality it follows%
\begin{equation}
\left.
\begin{array}
[c]{c}%
\left\Vert I_{2}\right\Vert _{L_{x,\xi}^{p}}\leq\dfrac{1}{\left\Vert
\phi\right\Vert _{L^{2}}^{2}}%
{\displaystyle\sum\limits_{j,k=1}^{N}}
{\displaystyle\sum\limits_{\left\vert \beta_{1}\right\vert +\left\vert
\beta_{2}\right\vert +\left\vert \beta_{3}\right\vert \leq2n}^{N}}
C_{\beta_{2}}\left\Vert \partial_{y}^{\beta_{1}}\phi_{jk}\right\Vert _{L^{1}%
}\left\Vert \partial_{y}^{\beta_{3}}\phi\right\Vert _{L^{1}}\left\Vert
\left\langle \cdot\right\rangle ^{-2n}\right\Vert _{L^{1}}\left\Vert u\left(
\tau,\cdot\right)  \right\Vert _{M^{p}}\\
\leq C_{T}\left\Vert u\left(  \tau,\cdot\right)  \right\Vert _{M^{p}},
\end{array}
\right.  \label{t18}%
\end{equation}
uniformly for $t\in\lbrack0,T]$ and $\tau\in\lbrack0,t].$ Since $\mathbf{V}%
_{2}$ satisfies (\ref{t23}) with $\left\vert \alpha\right\vert \geq0$,
similarly to (\ref{t18}) we prove that
\begin{equation}
\left.  \left\Vert I_{3}\right\Vert _{L_{x,\xi}^{p}}\leq C_{T}\left\Vert
u\left(  \tau,\cdot\right)  \right\Vert _{M^{p}},\right.  \label{t40}%
\end{equation}
uniformly for $t\in\lbrack0,T]$ and $\tau\in\lbrack0,t].$

At last, we estimate $I_{4}.$ Again, by (\ref{t6}), from (\ref{t10}) we have%
\[
R_{0}u\left(  x,\xi\right)  =\frac{1}{\left\Vert \phi\right\Vert _{L^{2}}^{2}}%
{\displaystyle\iiint\limits_{\mathbb{R}^{3N}}}
S\left(  x-y,\xi\right)  \phi\left(  y-z\right)  W_{\phi}u\left(
t,z,\eta\right)  e^{iy\cdot\left(  \eta-\xi\right)  }dz\text{\textit{\dj }%
}\eta dy.
\]
and then,%
\begin{equation}
\left\Vert I_{4}\right\Vert _{L_{x,\xi}^{p}}\leq\frac{1}{\left\Vert
\phi\right\Vert _{L^{2}}^{2}}\left\Vert
{\displaystyle\iiint\limits_{\mathbb{R}^{3N}}}
S\left(  x-y,g\left(  \tau;t,x,\xi\right)  \right)  \phi\left(  y-z\right)
W_{\phi}u\left(  \tau,z,\eta\right)  e^{iy\cdot\left(  \eta-g\left(
\tau;t,x,\xi\right)  \right)  }dz\text{\textit{\dj }}\eta dy\right\Vert
_{L_{x,\xi}^{p}}. \label{t12}%
\end{equation}
Using (\ref{t11}) in the right-hand side of (\ref{t12}) and integrating by
parts we get
\[
\left\Vert I_{4}\right\Vert _{L_{x,\xi}^{p}}\leq\frac{1}{\left\Vert
\phi\right\Vert _{L^{2}}^{2}}%
{\displaystyle\sum\limits_{\left\vert \beta_{1}\right\vert +\left\vert
\beta_{2}\right\vert \leq2n}}
\left\Vert
{\displaystyle\iiint\limits_{\mathbb{R}^{3N}}}
\left\vert \partial_{y}^{\beta_{1}}S\left(  x-y,g\left(  \tau;t,x,\xi\right)
\right)  \right\vert \left\vert \partial_{y}^{\beta_{2}}\phi\left(
y-z\right)  \right\vert \frac{\left\vert W_{\phi}u\left(  t,z,\eta\right)
\right\vert }{\left\langle \eta-g\left(  \tau;t,x,\xi\right)  \right\rangle
^{2n}}dz\text{\textit{\dj }}\eta dy\right\Vert _{L_{x,\xi}^{p}}.
\]
Making the change of variables $\Xi=g\left(  \tau;t,x,\xi\right)  ,$ similarly
to (\ref{t13}) we have%
\[
\left\Vert I_{4}\right\Vert _{L_{x,\xi}^{p}}\leq\frac{1}{\left\Vert
\phi\right\Vert _{L^{2}}^{2}}%
{\displaystyle\sum\limits_{\left\vert \beta_{1}\right\vert +\left\vert
\beta_{2}\right\vert \leq2n}}
\left(
{\displaystyle\iint\limits_{\mathbb{R}^{2N}}}
\left(
{\displaystyle\iiint\limits_{\mathbb{R}^{3N}}}
\left\vert \partial_{y}^{\beta_{1}}S\left(  x-y,\Xi\right)  \right\vert
\left\vert \partial_{y}^{\beta_{2}}\phi\left(  y-z\right)  \right\vert
\frac{\left\vert W_{\phi}u\left(  \tau,z,\eta\right)  \right\vert
}{\left\langle \eta-\Xi\right\rangle ^{2n}}dz\text{\textit{\dj }}\eta
dy\right)  ^{p}d\Xi dx\right)  ^{1/p}.
\]
Then, by Young's inequality we deduce%
\[
\left.
\begin{array}
[c]{c}%
\left\Vert I_{4}\right\Vert _{L_{x,\xi}^{p}}\leq\dfrac{1}{\left\Vert
\phi\right\Vert _{L^{2}}^{2}}\left\Vert \left\langle \cdot\right\rangle
^{-2n}\right\Vert _{L^{1}}\left\Vert \left\Vert W_{\phi}u\left(  \tau
,x,\xi\right)  \right\Vert _{L_{\xi}^{p}}\right\Vert _{L_{x}^{p}}%
{\displaystyle\sum\limits_{\left\vert \beta_{1}\right\vert +\left\vert
\beta_{2}\right\vert \leq2n}}
\left\Vert \left\Vert \partial_{y}^{\beta_{1}}S\left(  y,\xi\right)
\right\Vert _{L_{\xi}^{\infty}}\right\Vert _{L_{y}^{1}}\left\Vert \partial
_{y}^{\beta_{2}}\phi\left(  \cdot\right)  \right\Vert _{L^{1}}\\
\leq C\left\Vert u\left(  \tau,\cdot\right)  \right\Vert _{M^{p}}%
{\displaystyle\sum\limits_{\left\vert \beta_{1}\right\vert \leq2n}}
\left\Vert \left\Vert \partial_{y}^{\beta_{1}}S\left(  y,\xi\right)
\right\Vert _{L_{\xi}^{\infty}}\right\Vert _{L_{y}^{1}}.
\end{array}
\right.
\]
Then, using Lemma \ref{L2} we get%
\begin{equation}
\left\Vert I_{4}\right\Vert _{L_{x,\xi}^{p}}\leq C_{T}^{\left(  1\right)
}\left\Vert u\left(  \tau,\cdot\right)  \right\Vert _{M^{p}}, \label{t19}%
\end{equation}
uniformly for $t\in\lbrack0,T]$ and $\tau\in\lbrack0,t].$

Finally, using (\ref{t17}), (\ref{t18}) and (\ref{t19})\ in (\ref{t16}) we
arrive to%
\begin{equation}
\left\Vert u\left(  t,\cdot\right)  \right\Vert _{M^{p}}\leq\left\Vert
u_{0}\right\Vert _{M^{p}}+C%
{\displaystyle\int\limits_{0}^{t}}
\left\Vert u\left(  \tau,\cdot\right)  \right\Vert _{M^{p}}d\tau. \label{t21}%
\end{equation}
Applying Gronwall's lemma to (\ref{t21}), we attain (\ref{t20}). The case
$t\in\left[  -T,0\right]  $ can be considered similarly.
\end{proof}

The second assertion of Theorem \ref{TDirac} follows from:

\begin{theorem}
\label{T2}Let $1\leq p,q\leq\infty$ and $T>0.$ Set $m\in\mathbb{N}$. Let
$a\left(  \xi\right)  =\{a_{jk}\left(  \xi\right)  \}$ be an Hermitian
$\left(  m\times m\right)  $-matrix valued function such that $a_{jk}\left(
\xi\right)  \in C^{\infty}\left(  \mathbb{R}^{N}\right)  $ and the estimate
(\ref{t46}) is valid. Suppose that $\mathbf{V}\left(  t,x\right)
=\mathbf{V}_{1}\left(  t,x\right)  +\mathbf{V}_{2}\left(  t,x\right)  ,$ where
$\mathbf{V}_{1}\left(  t,x\right)  =\{\left(  \mathbf{V}_{1}\right)
_{jk}\left(  t,x\right)  \}$ is a matrix-valed function with $Q_{jk}\left(
t,x\right)  \in C^{\infty}\left(  \mathbb{R}^{N}\times\mathbb{R}^{N}\right)  $
that satisfies (\ref{t23}) for $\left\vert \alpha\right\vert \geq1.$ Moreover,
let $\mathbf{V}_{2}\left(  t,x\right)  \in C^{\infty}\left(  \mathbb{R}%
^{N}\times\mathbb{R}^{N}\right)  $ be a matrix-valued function that verifies
(\ref{t23}) for $\left\vert \alpha\right\vert \geq0.$ Then, the solution
$u\left(  t,x\right)  $ of (\ref{eq1})
\begin{equation}
\left\Vert u\left(  t,\cdot\right)  \right\Vert _{M^{p,q}}\leq C_{T}\left\Vert
u_{0}\right\Vert _{M^{p,q}}, \label{t22}%
\end{equation}
uniformly for $t\in\lbrack-T,T].$
\end{theorem}

\begin{proof}
We consider the case $t\in\left[  0,T\right]  .$ We expand the potential
$\boldsymbol{V}_{1}$ as
\[
\boldsymbol{V}_{1}\left(  t,y\right)  =\boldsymbol{V}_{1}\left(  t,x\right)
+\sum_{k=1}^{N}\left(  y_{k}-x_{k}\right)  \boldsymbol{V}_{k}\left(
t,y,x\right)  ,
\]
where%
\[
\boldsymbol{V}_{k}\left(  t,y,x\right)  :=%
{\displaystyle\int\limits_{0}^{1}}
\partial_{x_{k}}\boldsymbol{V}_{1}\left(  t,x+\theta\left(  y-x\right)
\right)  d\theta.
\]
Then,
\begin{equation}
\left.
\begin{array}
[c]{c}%
W_{\phi}\left(  qu\right)  \left(  t,x,\xi\right)  =\boldsymbol{V}_{1}\left(
t,x\right)  W_{\phi}u\left(  t,x,\xi\right)  +R_{1}u\left(  t,x,\xi\right)  ,
\end{array}
\right.  \label{t24}%
\end{equation}
where%
\[
R_{1}u\left(  t,x,\xi\right)  =\sum_{k=1}^{N}%
{\displaystyle\int\limits_{\mathbb{R}^{N}}}
e^{-iy\cdot\xi}\tilde{\phi}_{k}\left(  y-x\right)  \boldsymbol{V}_{k}\left(
t,y,x\right)  u\left(  t,y\right)  dy,
\]
and $\tilde{\phi}_{k:}\left(  y\right)  :=y_{k}\overline{\phi\left(  y\right)
}.$ By (\ref{t1}) and (\ref{t24}) we transform equation (\ref{eq1}) into
\begin{equation}
\left\{
\begin{array}
[c]{c}%
\left(  i\partial_{t}-\boldsymbol{V}_{1}\left(  t,x\right)  -a\left(
\xi\right)  \right)  W_{\phi}u\left(  t,x,\xi\right)  =R_{1}u\left(
x,\xi\right)  +\tilde{R}u\left(  x,\xi\right)  +R_{0}u\left(  x,\xi\right)
,\\
W_{\phi}\left(  u\right)  \left(  0,x,\xi\right)  =W_{\phi}\left(
u_{0}\right)  \left(  x,\xi\right)  .
\end{array}
\right.  \label{t25}%
\end{equation}
Thus, the solution of (\ref{t25}) is given by%
\[
W_{\phi}u\left(  t,x,\xi\right)  =e^{-ita\left(  \xi\right)  -i%
{\textstyle\int_{0}^{t}}
\boldsymbol{V}_{1}\left(  s,x\right)  ds}\left(  W_{\phi}\left(  u_{0}\right)
\left(  x,\xi\right)  -i%
{\displaystyle\int\limits_{0}^{t}}
e^{-i\tau a\left(  \xi\right)  -i%
{\textstyle\int_{0}^{\tau}}
\boldsymbol{V}_{1}\left(  s,x\right)  ds}\left(  R_{1}u+\tilde{R}%
u+R_{0}u\right)  \left(  \tau,x,\xi\right)  d\tau\right)  .
\]
Therefore,%
\begin{equation}
\left\vert W_{\phi}u\left(  t,x,\xi\right)  \right\vert \leq\left\vert
\tilde{I}_{1}\right\vert +%
{\displaystyle\int\limits_{0}^{t}}
\left(  \left\vert \tilde{I}_{2}\right\vert +\left\vert \tilde{I}%
_{3}\right\vert +\left\vert \tilde{I}_{4}\right\vert \right)  d\tau,
\label{t36}%
\end{equation}
where $\tilde{I}_{1}:=W_{\phi}\left(  u_{0}\right)  \left(  x,\xi\right)  ,$
$\tilde{I}_{2}:=R_{1}u\left(  \tau,x,\xi\right)  ,$ $\tilde{I}_{3}:=\tilde
{R}u\left(  \tau,x,\xi\right)  $ and $\tilde{I}_{4}:=R_{0}u\left(  \tau
,x,\xi\right)  .$ Let us estimate the norms $\left\Vert \tilde{I}%
_{j}\right\Vert _{M_{\phi}^{\infty,1}},$ $\left\Vert \tilde{I}_{j}\right\Vert
_{M_{\phi}^{1,\infty}}$ and $\left\Vert \tilde{I}_{j}\right\Vert _{M_{\phi
}^{p,p}},$ $1\leq p\leq\infty,$ for $j=2,3,4.$ By the inversion formula for
the wave-packet transform we have%
\begin{equation}
R_{1}u\left(  t,x,\xi\right)  =\frac{1}{\left\Vert \phi\right\Vert _{L^{2}%
}^{2}}%
{\displaystyle\sum\limits_{j,k=1}^{N}}
{\displaystyle\iiint\limits_{\mathbb{R}^{3N}}}
\tilde{\phi}_{k}\left(  y-x\right)  \boldsymbol{V}_{k}\left(  t,y,x\right)
\phi\left(  y-z\right)  W_{\phi}u\left(  t,z,\eta\right)  e^{iy\cdot\left(
\eta-\xi\right)  }dzd\eta dy. \label{t27}%
\end{equation}
Taking into account (\ref{t11}), integrating by parts in the right-hand side
of (\ref{t27}) and using (\ref{t23}) we estimate
\begin{equation}
\left\vert \tilde{I}_{2}\right\vert \leq\frac{1}{\left\Vert \phi\right\Vert
_{L^{2}}^{2}}%
{\displaystyle\sum\limits_{j,k=1}^{N}}
{\displaystyle\sum\limits_{\left\vert \beta_{1}\right\vert +\left\vert
\beta_{2}\right\vert +\left\vert \beta_{3}\right\vert \leq2n}}
C_{\beta_{2}}%
{\displaystyle\iiint\limits_{\mathbb{R}^{3N}}}
\left\vert \partial_{y}^{\beta_{1}}\tilde{\phi}_{k}\left(  y\right)
\right\vert \left\vert \partial_{y}^{\beta_{3}}\phi\left(  y+\left(
x-z\right)  \right)  \right\vert \frac{\left\vert W_{\phi}u\left(
t,z,\eta\right)  \right\vert }{\left\langle \eta-\xi\right\rangle ^{2n}%
}dzd\eta dy. \label{t28}%
\end{equation}
Then, by Young's inequality%
\begin{equation}
\left.
\begin{array}
[c]{c}%
\left\Vert \left\Vert \tilde{I}_{2}\right\Vert _{L_{x}^{\infty}}\right\Vert
_{L_{\xi}^{1}}\leq\frac{1}{\left\Vert \phi\right\Vert _{L^{2}}^{2}}\left\Vert
\left\langle \cdot\right\rangle ^{-2n}\right\Vert _{L^{1}}\left\Vert
\left\Vert W_{\phi}u\left(  t,x,\xi\right)  \right\Vert _{L_{x}^{\infty}%
}\right\Vert _{L_{\xi}^{1}}%
{\displaystyle\sum\limits_{j,k=1}^{N}}
{\displaystyle\sum\limits_{\left\vert \beta_{1}\right\vert +\left\vert
\beta_{2}\right\vert +\left\vert \beta_{3}\right\vert \leq2n}}
C_{\beta_{2}}\left\Vert \partial_{y}^{\beta_{1}}\tilde{\phi}_{k}\right\Vert
_{L^{1}}\left\Vert \partial_{y}^{\beta_{3}}\phi\right\Vert _{L^{1}}\\
\leq C_{T}\left\Vert u\left(  \tau,\cdot\right)  \right\Vert _{M_{\phi
}^{\infty,1}},
\end{array}
\right.  \label{t30}%
\end{equation}%
\begin{equation}
\left\Vert \left\Vert \tilde{I}_{2}\right\Vert _{L_{x}^{1}}\right\Vert
_{L_{\xi}^{\infty}}\leq C_{T}\left\Vert u\left(  \tau,\cdot\right)
\right\Vert _{M_{\phi}^{1,\infty}} \label{t31}%
\end{equation}
and%
\begin{equation}
\left\Vert \left\Vert \tilde{I}_{2}\right\Vert _{L_{x}^{p}}\right\Vert
_{L_{\xi}^{p}}\leq C_{T}\left\Vert u\left(  \tau,\cdot\right)  \right\Vert
_{M_{\phi}^{p,p}}, \label{t32}%
\end{equation}
uniformly for $t\in\lbrack0,T]$ and $\tau\in\lbrack0,t].$ Since $\mathbf{V}%
_{2}$ satisfies (\ref{t23}) with $\left\vert \alpha\right\vert \geq0$,
similarly to (\ref{t30})-(\ref{t32}) we prove that
\begin{equation}
\left.  \left\Vert \left\Vert \tilde{I}_{3}\right\Vert _{L_{x}^{\infty}%
}\right\Vert _{L_{\xi}^{1}}\leq C_{T}\left\Vert u\left(  \tau,\cdot\right)
\right\Vert _{M_{\phi}^{\infty,1}},\right.
\end{equation}%
\begin{equation}
\left\Vert \left\Vert \tilde{I}_{3}\right\Vert _{L_{x}^{1}}\right\Vert
_{L_{\xi}^{\infty}}\leq C_{T}\left\Vert u\left(  \tau,\cdot\right)
\right\Vert _{M_{\phi}^{1,\infty}}%
\end{equation}
and%
\begin{equation}
\left\Vert \left\Vert \tilde{I}_{3}\right\Vert _{L_{x}^{p}}\right\Vert
_{L_{\xi}^{p}}\leq C_{T}\left\Vert u\left(  \tau,\cdot\right)  \right\Vert
_{M_{\phi}^{p,p}}, \label{t44}%
\end{equation}
uniformly for $t\in\lbrack0,T]$ and $\tau\in\lbrack0,t].$ Next, using
(\ref{t11}) and integrating by parts in the right-hand side of (\ref{t29}) we
have%
\[
\left\vert \tilde{I}_{4}\right\vert \leq\frac{1}{\left\Vert \phi\right\Vert
_{L^{2}}^{2}}%
{\displaystyle\sum\limits_{\left\vert \beta_{1}\right\vert +\left\vert
\beta_{2}\right\vert \leq2n}}
{\displaystyle\iiint\limits_{\mathbb{R}^{3N}}}
\left\vert \partial_{y}^{\beta_{1}}S\left(  x-y,\xi\right)  \right\vert
\left\vert \partial_{y}^{\beta_{2}}\phi\left(  y-z\right)  \right\vert
\frac{\left\vert W_{\phi}u\left(  t,z,\eta\right)  \right\vert }{\left\langle
\eta-\xi\right\rangle ^{2n}}dzd\eta dy.
\]
Moreover, using Lemma \ref{L2} we get%
\[
\left\vert \tilde{I}_{4}\right\vert \leq\frac{C}{\left\Vert \phi\right\Vert
_{L^{2}}^{2}}%
{\displaystyle\sum\limits_{\left\vert \beta_{1}\right\vert +\left\vert
\beta_{2}\right\vert \leq2n}}
{\displaystyle\iiint\limits_{\mathbb{R}^{3N}}}
\left\vert \left\langle x-y\right\rangle ^{-2n}\right\vert \left\vert
\partial_{y}^{\beta_{2}}\phi\left(  y-z\right)  \right\vert \frac{\left\vert
W_{\phi}u\left(  t,z,\eta\right)  \right\vert }{\left\langle \eta
-\xi\right\rangle ^{2n}}dzd\eta dy.
\]
Then, by Young's inequality%
\begin{equation}
\left.  \left\Vert \left\Vert \tilde{I}_{4}\right\Vert _{L_{x}^{\infty}%
}\right\Vert _{L_{\xi}^{1}}\leq C_{T}\left\Vert u\left(  \tau,\cdot\right)
\right\Vert _{M_{\phi}^{\infty,1}},\right.  \label{t33}%
\end{equation}
and%
\begin{equation}
\left\Vert \left\Vert \tilde{I}_{4}\right\Vert _{L_{x}^{1}}\right\Vert
_{L_{\xi}^{\infty}}\leq C_{T}\left\Vert u\left(  \tau,\cdot\right)
\right\Vert _{M_{\phi}^{1,\infty}}. \label{t34}%
\end{equation}
and%
\begin{equation}
\left\Vert \left\Vert \tilde{I}_{4}\right\Vert _{L_{x}^{p}}\right\Vert
_{L_{\xi}^{p}}\leq C_{T}\left\Vert u\left(  \tau,\cdot\right)  \right\Vert
_{M_{\phi}^{p,p}}, \label{t42}%
\end{equation}
uniformly for $t\in\lbrack0,T]$ and $\tau\in\lbrack0,t].$ Using estimates
(\ref{t30})-(\ref{t44}), (\ref{t33})-(\ref{t42}) in (\ref{t36}) we get%
\[
\left\Vert u\left(  \tau,\cdot\right)  \right\Vert _{M_{\phi}^{\infty,1}}\leq
C\left\Vert u_{0}\right\Vert _{M_{\phi}^{\infty,1}}+C%
{\displaystyle\int\limits_{0}^{t}}
\left\Vert u\left(  \tau,\cdot\right)  \right\Vert _{M_{\phi}^{\infty,1}}%
d\tau,
\]%
\[
\left\Vert u\left(  \tau,\cdot\right)  \right\Vert _{M_{\phi}^{1,\infty}}\leq
C\left\Vert u_{0}\right\Vert _{M_{\phi}^{1,\infty}}+C%
{\displaystyle\int\limits_{0}^{t}}
\left\Vert u\left(  \tau,\cdot\right)  \right\Vert _{M_{\phi}^{1,\infty}}%
d\tau,
\]
and%
\[
\left\Vert u\left(  \tau,\cdot\right)  \right\Vert _{M_{\phi}^{p,p}}\leq
C\left\Vert u_{0}\right\Vert _{M_{\phi}^{p,p}}+C%
{\displaystyle\int\limits_{0}^{t}}
\left\Vert u\left(  \tau,\cdot\right)  \right\Vert _{M^{p}}d\tau.
\]
Then, Gronwall's lemma yields
\begin{equation}
\left\Vert u\left(  \tau,\cdot\right)  \right\Vert _{M_{\phi}^{\infty,1}}\leq
C\left\Vert u_{0}\right\Vert _{M_{\phi}^{\infty,1}}, \label{t37}%
\end{equation}%
\begin{equation}
\left\Vert u\left(  \tau,\cdot\right)  \right\Vert _{M_{\phi}^{1,\infty}}\leq
C\left\Vert u_{0}\right\Vert _{M_{\phi}^{1,\infty}}, \label{t38}%
\end{equation}
and%
\begin{equation}
\left\Vert u\left(  \tau,\cdot\right)  \right\Vert _{M_{\phi}^{p,p}}\leq
C\left\Vert u_{0}\right\Vert _{M_{\phi}^{p,p}}. \label{t39}%
\end{equation}
Estimate (\ref{t22}) follows from (\ref{t37})-(\ref{t39}) by the complex
interpolation theorem for modulation spaces (see Proposition
\ref{PropInterpolation}). The case $t\in\left[  -T,0\right]  $ can be
considered similarly.
\end{proof}

\end{document}